\documentclass[11pt]{amsart}

\usepackage{amssymb,amscd,amsmath,color,mathabx,mathtools}
\usepackage[all,cmtip]{xy}
\usepackage{mathrsfs}
\usepackage{lipsum}
\usepackage{nicematrix,tikz, ifthen}
\usepackage{scalerel}
\usepackage{ascii}

\usepackage{extarrows}

\usepackage[shortlabels]{enumitem}
\usepackage{tensor}

\usepackage[T1]{fontenc}

\usepackage[colorlinks,
linkcolor=black!75!red,
citecolor=blue,
pdftitle={},
pdfproducer={pdfLaTeX},
bookmarksopen=true,
bookmarksnumbered=true,
backref=page]{hyperref}

\makeatletter
\newenvironment{claimproof}{\par
  \pushQED{\hfill$\diamondsuit$}%
  \normalfont \topsep6\p@\@plus6\p@\relax
  \trivlist
  \item[\hskip\labelsep
    \normalfont\itshape 
    Proof.\hspace{-0.1cm}
    ]\ignorespaces
}{%
  \popQED\endtrivlist\@endpefalse
}
\makeatother

\usepackage[margin=3cm]{geometry}

\newcommand{\N}{{\ensuremath{\mathbb{N}}}}
\newcommand{\Z}{{\ensuremath{\mathbb{Z}}}}

\newcommand{\C}{{\ensuremath{\mathbb{C}}}}

\usepackage[normalem]{ulem}

\newcommand{\stkout}[1]{\ifmmode\text{\sout{\ensuremath{#1}}}\else\sout{#1}\fi}

\DeclareMathOperator{\GL}{GL}
 
\DeclareMathOperator{\spc}{sp}

\DeclareMathOperator{\supp}{supp}

\DeclareMathOperator{\Ad}{Ad}

\DeclareMathOperator{\sing}{Sing}

\DeclareMathOperator{\diag}{diag}
\DeclareMathOperator{\G}{G}
\DeclareMathOperator{\U}{U}

\DeclareMathOperator{\SL}{SL}
\DeclareMathOperator{\Tr}{Tr}

\newcommand\bC{{\mathbb C}}

\newcommand\bZ{{\mathbb Z}}

\newcommand\cN{{\mathcal N}}

\newcommand{\Sph}{\ensuremath{\mathbb{S}}}

\newcommand{\ca}[1]{\ensuremath{\mathcal{#1}}}

\newcommand{\abs}[1]{\ensuremath{ {\left| #1 \right|} }}

\newtheorem{proposition}{Proposition}[section]
\newtheorem{lemma}[proposition]{Lemma}

\newtheorem{theorem}[proposition]{Theorem}
\newtheorem{theoremrec}{Theorem}

\newtheorem{claim}{Claim}
\newtheorem{corollary}[proposition]{Corollary}
\theoremstyle{definition}

\newtheorem{remarks}[proposition]{Remarks}

\newtheorem{example}[proposition]{Example}
\newtheorem{examples}[proposition]{Examples}

\newtheorem{case}{Case}

\AtBeginEnvironment{remark}{%
  \pushQED{\qed}%
}
\AtEndEnvironment{remark}{\popQED\endremark}

\AtBeginEnvironment{remarks}{%
  \pushQED{\qed}%
}
\AtEndEnvironment{remarks}{\popQED\endremarks}

\AtBeginEnvironment{example}{%
  \pushQED{\qed}%
}
\AtEndEnvironment{example}{\popQED\endexample}

\AtBeginEnvironment{examples}{%
  \pushQED{\qed}%
}
\AtEndEnvironment{examples}{\popQED\endexamples}

\newtheorem{remark}[proposition]{Remark}

\numberwithin{equation}{section}

\newlength{\leftstackrelawd}
\newlength{\leftstackrelbwd}
\def\leftstackrel#1#2{\settowidth{\leftstackrelawd}%
{${{}^{#1}}$}\settowidth{\leftstackrelbwd}{$#2$}%
\addtolength{\leftstackrelawd}{-\leftstackrelbwd}%
\leavevmode\ifthenelse{\lengthtest{\leftstackrelawd>0pt}}%
{\kern-.5\leftstackrelawd}{}\mathrel{\mathop{#2}\limits^{#1}}}

\makeatletter
\@namedef{subjclassname@2020}{\textup{2020} Mathematics Subject Classification}
\makeatother

\begin{document}

\title[A variant of \v{S}emrl's preserver theorem for singular matrices]{A variant of \v{S}emrl's preserver theorem for singular matrices}

\author{Alexandru Chirvasitu, Ilja Gogi\'{c}, Mateo Toma\v{s}evi\'{c}}

\address{A.~Chirvasitu, Department of Mathematics, University at Buffalo, Buffalo, NY 14260-2900, USA}
\email{achirvas@buffalo.edu}

\address{I.~Gogi\'c, Department of Mathematics, Faculty of Science, University of Zagreb, Bijeni\v{c}ka 30, 10000 Zagreb, Croatia}
\email{ilja@math.hr}

\address{M.~Toma\v{s}evi\'c, Department of Mathematics, Faculty of Science, University of Zagreb, Bijeni\v{c}ka 30, 10000 Zagreb, Croatia}
\email{mateo.tomasevic@math.hr}



\date{\today}

\begin{abstract}
  For positive integers $1 \leq k \leq n$ let $M_n$ be the algebra of all $n \times n$ complex matrices and $M_n^{\le k}$ its subset consisting of all matrices of rank at most $k$.  We first show that whenever $k>\frac{n}{2}$, any continuous spectrum-shrinking map $\phi : M_n^{\le k} \to M_n$ (i.e. $\mathrm{sp}(\phi(X)) \subseteq \mathrm{sp}(X)$ for all $X \in M_n^{\le k}$) either preserves characteristic polynomials or takes only nilpotent values. Moreover, for any $k$ there exists a real analytic embedding of $M_n^{\le k}$ into the space of $n\times n$ nilpotent matrices for all sufficiently large $n$. This phenomenon cannot occur when $\phi$ is injective and either $k > n - \sqrt{n}$ or the image of $\phi$ is contained in $M_n^{\le k}$. We then establish a main result of the paper -- a variant of \v{S}emrl's preserver theorem for $M_n^{\le k}$: if $n \geq 3$, any injective continuous map $\phi :M_n^{\le k} \to M_n^{\le k}$ that preserves commutativity and shrinks spectrum is of the form $\phi(\cdot)=T(\cdot)T^{-1}$ or $\phi(\cdot)=T(\cdot)^tT^{-1}$, for some invertible matrix $T\in M_n$. Moreover, when $k=n-1$, which corresponds to the set of singular $n\times n$ matrices, this result extends to maps $\phi$ which take values in $M_n$. Finally, we discuss the indispensability of assumptions in our main result.
\end{abstract}

\maketitle

\section{Introduction}

This is a follow-up on prior work ultimately motivated by the problem of characterizing \emph{Jordan morphisms} \cite[\S I.1]{jac_jord} between matrix algebras or more broadly, operator algebras: linear maps satisfying
\begin{equation*}
  \phi(XY+YX)
  =
  \phi(X)\phi(Y)+\phi(Y)\phi(X)
  ,\quad
 \text{ for all } X,Y.
\end{equation*}
We remind the reader (e.g. \cite{Herstein,Semrl2}) that such non-zero self-maps on the $n\times n$ complex matrix algebra $M_n$ are classified as precisely those of the form
\begin{equation}\label{eq:inner}
  \phi(\cdot) = T(\cdot)T^{-1} \qquad \text{or}   \qquad \phi(\cdot) = T(\cdot)^{t}T^{-1},
\end{equation}
for some invertible matrix $T \in M_n$.

The vast literature on Jordan morphisms between operator algebras has branched out in numerous directions and spawned many variations on the theme of recognizing, characterizing and classifying such morphisms via various (linear and nonlinear) preservation properties. One notable example of such a result is the theorem alluded to in the present paper's title:
\begin{theoremrec}[\v{S}emrl's preserver theorem {\cite[Theorem~1.1]{Semrl}}]\label{th:Semrl}
  Let $\phi : M_n \to M_n$, $n \ge 3$, be a continuous map that preserves
  \begin{itemize}[wide]
  \item commutativity, in the sense that
    \begin{equation*}
      XY=YX \quad \implies \quad \phi(X)\phi(Y)=\phi(Y)\phi(X), \quad \text{ for all } X,Y \in M_n; 
    \end{equation*}
  \item and also spectra, i.e. $\spc(\phi(X))=\spc(X)$ for all $X \in M_n$.     
  \end{itemize}
  Then, there exists an invertible matrix $T \in M_n$ such that $\phi$ is as in \eqref{eq:inner}.
\end{theoremrec}

The version just cited improves on earlier work \cite{PetekSemrl} and employs relatively sophisticated techniques relying, for instance, on a variant of the \emph{Fundamental Theorem of Projective Geometry} (see e.g.\ \cite[Theorem 3.1]{zbMATH01747827} or \cite[Theorem 2.3]{pank_wign}); this allows for a more concise argument as compared to prior approaches based on elementary calculations. \cite{Semrl} also gives examples demonstrating the indispensability of the various assumptions. For further related results, ramifications and generalizations we direct the reader to \cite{2501.06840v2,GogicPetekTomasevic,MR4881574,Petek-HM,Petek-TM} and the references cited therein.

In the same spirit of generalizing characterization or classification results on and around preserver maps, in \cite{2501.06840v2} we show that in Theorem \ref{th:Semrl} it suffices to assume that $\phi$ is merely spectrum-\emph{shrinking} (meaning that $\spc(\phi(X)) \subseteq \spc(X)$ for all matrices $X \in M_n$) rather than spectrum-preserving. More generally, given a subset $L\subseteq \C^n$, denote by $\Delta_{L}$ the subset of $L$ that consists of elements with at least two equal coordinates. We naturally identify the symmetric group $S_n$ with the $n\times n$ permutation matrices, so that $S_n$ forms a subgroup of the general linear group $\GL(n)$. Assuming that $L$ is invariant under the action of $S_n$ in $\C^n$ (by conjugation), $S_n$ also naturally acts on the set of connected components of $L\setminus \Delta_L$.  We have:
\begin{theorem}[{\cite[Theorem~1.1]{2501.06840v2}}]\label{th:sp.shrk.conn.conf.sp}
  Given $n \in \Z_{\ge 1}$, a closed connected subgroup $\G$ of $\GL(n)$, a linear subspace $V$ of the algebra $T^{+}(n)$ of $n \times n$ strictly upper-triangular matrices, and a subset $L \subseteq \C^n$, denote
  \begin{equation*}
  T_{L,V}:=\{\diag(\lambda_1, \ldots, \lambda_n) + v \ : \ (\lambda_1, \ldots, \lambda_n) \in L, \ v \in V\}
\end{equation*}
and 
\begin{equation*}
 \ca{X}_n:=\Ad_{\G}T_{L,V}=\{SXS^{-1} \, : \, X \in T_{L,V}, \, S \in \G\}.
 \end{equation*}
Assume that: 
  \begin{itemize}[wide]
  \item $L\setminus \Delta_L$ is dense in $L$;
  \item $L$ is invariant under the action of $S_n$ in $\C^n$;
  \item and  the isotropy groups of the connected components of $L\setminus \Delta_L$ in $\G\cap S_n$ are transitive on $\{1,2,\ldots,n\}$.
  \end{itemize} 
  Then for an arbitrary $m\in \bZ_{\ge 1}$ there exists a continuous spectrum shrinker $\phi :   \ca{X}_n \to M_m$
if and only $n$ divides $m$ and in that case we have the equality of characteristic polynomials
  \begin{equation}\label{eq:k.is.pow}
    k_{\phi(X)}
    =
    (k_{X})^{\frac mn}
    ,\quad
   \text{ for all }  X\in \ca{X}_n.
  \end{equation}
\end{theorem}
As established in \cite[Corollary~1.2]{2501.06840v2}, Theorem \ref{th:sp.shrk.conn.conf.sp} applies to a wide array of distinguished matrix subsets $\ca{X}_n$ of $M_n$, including: $M_n$, $\GL(n)$, the special linear group $\SL(n)$, the unitary group $\U(n)$ and the subset $N_n$ of $n \times n$ normal matrices. Additionally, it also holds for the sets of semisimple (i.e.\ diagonalizable) elements in $M_n$, $\GL(n)$ and $\SL(n)$. Moreover, by synthesizing \ \cite[Corollary~1.2]{2501.06840v2} with a range of algebraic, topological, and operator-theoretic techniques, in \cite[Theorem~0.2]{2501.06840v2} we show that for
\begin{equation*}
  \ca{X}_n
  \
  \in
  \ 
  \{\GL(n), \SL(n), \U(n), N_n\}
\end{equation*}
any continuous, commutativity-preserving and spectrum-shrinking map $\phi: \ca{X}_n \to M_n$ is of the form \eqref{eq:inner} for some $T \in \GL(n)$. 

\smallskip

The primary objective of this paper is to establish a variant of Theorem \ref{th:Semrl} and \cite[Theorem~0.2]{2501.06840v2} for a class of maps $\phi$ defined on appropriate subsets of all \emph{singular matrices} $\sing(n)\subset M_n$.  More precisely, if $k \leq n$, let
\begin{equation*}
  M_n^{=k}:=\{A \in M_n \ : \ r(A)= k\}, \quad \mbox{ and } \quad M_n^{\le k}:=\{A \in M_n \ : \ r(A)\leq k\},
\end{equation*}
where $r(A)$ denotes the rank of a matrix $A$. Obviously, $\sing(n)=M_n^{\le n-1}$.

Note that Theorem \ref{th:sp.shrk.conn.conf.sp} does not apply to $M_n^{\le k}$ when $k\le n-1$. Specifically, if $k \le n-2$, the space $L$ of possible $n$-tuples of eigenvalues would always contain at least two zeros, implying that $L\setminus \Delta_L$ are empty, so certainly not dense in $L$. Additionally, for $k=n-1$, the connected components of $L\setminus \Delta_L$ are 
\begin{equation*}
  \left\{(\lambda_1, \ldots , \lambda_n)\in \bC^n\ :\ \lambda_1, \ldots, \lambda_n\text{ pairwise distinct and }\lambda_j=0\right\}
  ,\quad
  1\le j\le n.
\end{equation*}
The respective isotropy groups are the copies of $S_{n-1}\subset S_n$ fixing the $n$ individual symbols, so do not operate transitively. These issues notwithstanding, we prove the following singular variant:

\begin{theorem}\label{th:sing}
  Let $m,n,k \in \Z_{\ge 1}$, $k \leq n$, and let $\phi : M_n^{\le k}\to M_m$ be a continuous spectrum-shrinking map. 
  \begin{enumerate}[(1),wide]
  \item\label{item:th:sing:p} There is some $p\in \bZ_{\ge 0}$ such that
    \begin{equation}\label{eq:somep}
      \text{ for all } X\in M_{n}^{\le k}
      \quad:\quad
      k_{\phi(X)}(x)
      =
      k_X^p(x)\cdot x^{m-pn}.
    \end{equation}
  \item\label{item:th:sing:p01} If $2k > m$ then $\phi$ either takes nilpotent values or
    \begin{equation}\label{eq:samepoly+leftover0}
    \text{ for all } X\in M_{n}^{\le k}
      \quad:\quad
      k_{\phi(X)}(x)
      =
      k_X(x)\cdot x^{m-n}.
    \end{equation}
    In particular, if $m=n$, then $\phi$ either takes only nilpotent values or preserves characteristic polynomials.

  \item\label{item:th:sing:inj} If in addition to the preceding point's conditions we also have $k(2n-k)> m^2-m$ and $\phi$ is injective, then \eqref{eq:samepoly+leftover0} holds.
    
  \item\label{item:th:sing:inj.bis} If $m=n$, $\phi$ is injective, and $k > n - \sqrt{n}$ or $\phi(M_n^{\le k}) \subseteq M_n^{\le k}$, then $\phi$ preserves characteristic polynomials (and hence also spectra).
  \end{enumerate}
\end{theorem}
Additionally, we provide an argument that for any fixed positive integer $k$ there exist (real analytic) embeddings of $M_n^{\le k}$ into the space of all $n\times n$ nilpotent matrices, provided $n$ is sufficiently large (Remark \ref{re:smlk.lgn}).

Using Theorem \ref{th:sing} combined with the techniques developed in \cite{MR4881574}, we then prove our main result -- a variant of Theorem \ref{th:Semrl} for sets $M_n^{\le k}$:
\begin{theorem}\label{th:sing2sing}
  Let $n \ge 3$ and $1 \le k < n$. If $\phi : M_n^{\le k} \to M_n^{\le k}$ is an injective continuous commutativity-preserving and spectrum-shrinking map, then there exists $T \in \mathrm{GL}(n)$ such that $\phi$ is of the form \eqref{eq:inner}.
\end{theorem}
In contrast to the previously discussed results, the injectivity condition in Theorem \ref{th:sing2sing} proves to be indispensable (see Remark \ref{re:indispensable}). Further, when $k=n-1$, leveraging the fact that any injective continuous spectrum-shrinking map $\phi : \sing(n)\to M_n$ automatically preserves  characteristic polynomials (Theorem \ref{th:sing}\ref{item:th:sing:inj.bis}), so that in particular $\phi(\sing(n)) \subseteq \sing(n)$, we obtain the following consequence of Theorem \ref{th:sing2sing}, for the case when $\phi$ takes values in $M_n$:
\begin{corollary}\label{cor:singular}
If $\phi :\sing(n) \to M_n$, $n \ge 3$, is an injective continuous commutativity-preserving and spectrum-shrinking map, then there exists $T \in \mathrm{GL}(n)$ such that $\phi$ is of the form \eqref{eq:inner}.
\end{corollary}
The proofs of Theorems \ref{th:sing} and \ref{th:sing2sing} will be provided in the next section.  We further remark that for $k=1$, Theorem \ref{th:sing2sing} cannot be strengthened for maps with values in $M_n$, as the following example demonstrates:
$$
\phi : M_{n}^{\le 1} \to M_n, \qquad \phi(X):=\Tr(X)I - X
$$
(where $\Tr(X)$ denotes the trace of $X$). Clearly, $\phi$ is a linear (hence continuous) injective map that preserves both spectrum and commutativity, yet it does not assume the form of \eqref{eq:inner}.  However,  we are currently uncertain whether Theorem \ref{th:sing2sing} holds for maps $\phi :  M_n^{\le k} \to M_n$  when $2 \le k \le n-2$  and we anticipate addressing this question in future research. 

\subsection*{Acknowledgments}
We sincerely appreciate the referee's thorough review of the paper and the valuable insights and suggestions they offered.

\section{Proofs of Theorems \ref{th:sing} and \ref{th:sing2sing}}\label{se:sing.sp.pres}
Before proving  Theorem \ref{th:sing} we need some preparation. We begin by citing the next simple, yet useful lemma from \cite{2501.06840v2}.
\begin{lemma}[Lemma~1.7 of \cite{2501.06840v2}]\label{le:topological lemma}
  Let $X$ and  $Y$ be topological spaces such that $X$ is connected and $Y$ is Hausdorff. Let $n\in \N$ and suppose that $g,f_1,\ldots,f_n : X \to Y$ are continuous functions such that for each $x \in X$ we have
  \begin{itemize}[wide]
  \item $g(x) \in \{f_1(x),\ldots,f_n(x)\}$,
  \item $f_i(x) \ne f_j(x)$ for all $1 \le i \ne j \le n$.
  \end{itemize}
  Then there exists $1 \leq i \leq n$ such that $g=f_i$.
\end{lemma}
Let us now introduce some notation which will be used throughout the paper. First of all, by ${\mathbb{C}}^{\times}$ we denote the multiplicative group of non-zero complex numbers. For an integer $n\in {\mathbb{Z}_{\ge 1}}$ we write 
\begin{equation*}
[n]:=\{1,\ldots, n\},
\end{equation*}
and for any subset $J \subseteq [n]$, we denote its cardinality by $|J|$. For each $0 \leq k \leq n$, we let 
 \begin{equation*}
  [n]^{(k)}:=
          \left\{J \subseteq [n] \ : \ |J|=k\right\},
\end{equation*}
to denote the family of all $k$-element subsets of $[n]$.

\smallskip

As usual, we identify vectors in ${\mathbb C}^n$ with the corresponding column-matrices. Given a matrix $A \in M_n$, we denote its characteristic polynomial by
\begin{equation*}
k_A(x) := \det(x I-A).
\end{equation*}
Additionally, for $i,j \in [n]$ we denote by $E_{ij}\in M_n$ the standard matrix unit with $1$ at the position $(i,j)$ and $0$ elsewhere. Similarly, the canonical basis vectors of ${\mathbb C}^n$ are denoted by $e_1, \ldots, e_n$.

\smallskip

Let $\mathcal{N}(n)$ denote the set of all nilpotent matrices in $M_n$, and let $\mathcal{D}_n \subset M_n$ be the subalgebra of diagonal matrices. For each integer $0 \leq k \leq n$, denote by
\begin{equation*}
\mathcal{D}_n^{\le k}:= \mathcal{D}_n\cap M_n^{\leq k},
\end{equation*}
the set of $n \times n$ diagonal matrices of rank at most $k$. Given any subset $\mathcal{S}\subseteq \mathcal{D}_n$ and any  $J\in   [n]^{(k)}$, set
\begin{equation*}
 \mathcal{S}_{J}:=
          \left\{\diag(\lambda_1, \ldots, \lambda_n)\in \mathcal{S} \ :\   \lambda_j=0 \text{ for all } j \in J\right\},         
\end{equation*}
so that $\mathcal{S}_{J}\subseteq \mathcal{D}_n^{\le n-k}$.

\smallskip

In what follows, unless explicitly stated otherwise, we fix integers $n\in {\mathbb{Z}_{\ge 2}}$ and $k \in [n]$.

\begin{lemma}\label{le:diagonal change of nilpotent Jordan block}
  Let $S \subseteq [n-1]$ and $A := \sum_{j \in S}E_{j,j+1} \in M_n$. For each $j \in S$ choose some $\lambda_j \in \C^\times$. Then the matrix 
  \begin{equation*}B := \sum_{j \in S}(\lambda_jE_{jj} + E_{j,j+1}) \in M_n\end{equation*}
  satisfies $r(B) = r(A)$.
\end{lemma}
\begin{proof}
  An elementary linear algebra exercise.
\end{proof}
Throughout,  the abbreviation ``p.\ d.'' stands for ``pairwise distinct''.

\smallskip
 For a matrix $X = [X_{ij}]_{i,j=1}^n \in M_n$ and a  vector $v = (v_1,\ldots,v_n) \in \C^n$, we define their respective \emph{supports} as 
\begin{equation*}
\supp X:=\{(i,j)\in [n]^2 : X_{ij}\ne 0\} \quad \text{and} \quad \supp v: = \{ j \in [n] : v_j \ne 0\}.
\end{equation*}
\begin{lemma}\label{le:Rk density argument}\phantom{x}
  \begin{enumerate}[(a),wide]
  \item The set
    \begin{equation*}\ca{S} := \{X \in M_n^{=k} \ : \ k_X(x) = (x-\lambda_1)\cdots (x-\lambda_k)x^{n-k} \text{ with p.\ d.\ }\lambda_1,\ldots,\lambda_k \in \C^\times\}\end{equation*}
    is dense in $M_n^{=k}$.
  \item Suppose that $A \in M_n^{=k}$ satisfies $\spc(A)= \{0,\lambda_1,\ldots,\lambda_k\}$ for some distinct $\lambda_1,\ldots,\lambda_k \in \C^\times$. Then $A$ is diagonalizable.
  \end{enumerate}
\end{lemma}
\begin{proof}
  \begin{enumerate}[(a),wide]
  \item Let $A \in M_n^{=k}$ be arbitrary. Without loss of generality, assume that $A$ equals its Jordan form, which takes the shape
    \begin{equation*}A = \diag(J(\lambda_1)_{p_1\times p_1},\ldots,J(\lambda_r)_{p_r\times p_r},J(0))\end{equation*}
    for some distinct $\lambda_1,\ldots,\lambda_r \in \C^\times$, where each $J(\lambda)_{p\times p}$ denotes a $p\times p$ Jordan matrix with $\lambda$-s on the diagonal. We have
    \begin{equation*}k = r(A) = p_1 + \cdots + p_r + r(J(0)).\end{equation*}
    Let $B \in M_n$ be any matrix whose off-diagonal elements match those of $A$, while its diagonal consists of distinct non-zero complex numbers on the positions $[p_1+\cdots+p_r]$ and $\{p_1+\cdots+p_r + j : (j,j+1) \in \supp J(0)\}$. Then $B$ retains the block-diagonal structure of $A$, and its first $r$ blocks are invertible matrices (just like those of $A$), while its last block has the same rank as $J(0)$ by Lemma \ref{le:diagonal change of nilpotent Jordan block}. It follows that $r(B) = r(A) = k$. We conclude that any such matrix $B$ in fact lies in $\ca{S}$. It is evident we can construct such a $B$ arbitrarily close to $A$.

     \smallskip

  \item Since the nullity of $A$ is precisely $n-k$, there are that many elementary $0$-blocks in the Jordan form of $A$. 
  \end{enumerate}
  
\end{proof}
\begin{lemma}\label{le:open to open}
  Suppose that $\phi : M_n^{=k} \to M_n^{\le k}$ is an injective continuous map. Then there exists an open subset $U \subseteq M_n^{=k}$ such that $\phi(U)$ is an open subset of  $M_n^{=k}$.
\end{lemma}
\begin{proof}
 First recall \cite[Example 14.16]{har_ag} that for any $k\in [n]$, $M_n^{=k}$ is an irreducible, smooth, complex algebraic variety of dimension $k(2n-k)$, and hence also a smooth connected \emph{real} manifold of dimension $2k(2n-k)$ and that $M_n^{=k}$ is open in $M_n^{\le k}$ (by the lower semicontinuity of the rank). Let
    \begin{equation*}
      m:=\max_{A \in M_n^{=k}} r(\phi(A)).
    \end{equation*}
    If $m <k$ then, by continuity, $\phi$ would map some open subset of $M_n^{=k}$ into $M_{n}^{=m}$, which would contradict the \emph{Invariance of Domain theorem} \cite[Corollary IV.19.9]{bred_gt_1997}. Hence, $m=k$, so there exists an open subset $U \subseteq M_n^{=k}$ such that $\phi(U)\subseteq M_n^{=k}$. By another application of domain invariance, we obtain that $\phi(U)$ is open in $M_n^{=k}$.
\end{proof}

\begin{proof}[Proof of Theorem \ref{th:sing}] We address the claims separately. 
  
  \begin{enumerate}[label={},wide]

  \item \textbf{\ref{item:th:sing:p}} Denote by $(M_{n}^{= k})'_{s}\subset M_{n}^{\le k}$ the collection of all semisimple (rank-$k$) matrices with precisely $k$ non-zero distinct eigenvalues, and by $(\ca{D}^{=k}_n)':=(M_{n}^{= k})_{s}\cap \ca{D}_n$ that of diagonal matrices therein (since diagonal matrices are semisimple, we omit the subscript ``s''.). Because the continuous polynomial-valued map $k_{\phi(\cdot)}$ is constant on every (connected!) conjugacy class, it is enough to examine its effect on $(\ca{D}_n^{= k})'$, as $\Ad_{\GL(n)}(\ca{D}_n^{= k})'=(M_{n}^{=k})_{s}'$ and the latter set is dense in $M_{n}^{\le k}$. In terms of $(\ca{D}_n^{= k})'$, the statement translates to the claim that all non-zero eigenvalues of $X\in (\ca{D}_n^{=k})'$ have the same algebraic multiplicity $p$ for $\phi(X)$.
    
Following our notation, consider the sets
    \begin{equation}\label{eq:def.dnj}
      (\ca{D}_n^{= k})'_{J}
        =
          \left\{\diag(\lambda_1, \ldots, \lambda_n)\in (\ca{D}_n^{= k})'\ :\ \lambda_j=0\iff j\in J\right\}
          ,\quad
          J\in   [n]^{(n-k)}.
    \end{equation}
Note that each space  $(\ca{D}_n^{= k})'_{J}$ is path-connected, as it is homeomorphic to the $k$-th configuration space 
    \begin{equation*}
        \ca{C}^k(\C^\times) = \{(z_1,\ldots,z_k) \in (\C^\times)^k \ :\ z_i \ne z_j \text{ for all } i \ne j\},
    \end{equation*}
    of $\C^\times$ (see e.g.\ \cite[Definition 1.1]{zbMATH05785888}). Moreover, the (path-)connected components of $(\ca{D}_n^{=k})'$ are precisely given by  \eqref{eq:def.dnj}. The isotropy subgroup $S_{n,(\ca{D}_n^{= k})'_{J}}\le S_n$ of $(\ca{D}_n^{= k})'_{J}$ in $S_n\le \GL(n)$ is the intersection of the isotropy groups of all $j\in J$, and hence permutes the $k$ elements of $[n]\setminus J$ transitively. It follows that on $(\ca{D}_n^{= k})'_{J}$ the polynomial $k_{\phi(\cdot)}$ is symmetric in $\lambda_j$, $j\in [n]\setminus J$: for any permutation $\sigma\in S_n$ fixing $J$ pointwise and $D=\diag(\lambda_1, \ldots, \lambda_n)\in (\ca{D}_n^{= k})'_{J}$ we have 
    \begin{equation*}
      \begin{aligned}
        x^{m-\sum p_j}\prod_{j\not\in J}(x-\lambda_j)^{p_j}
        &=
          k_{\phi(D)}(x)
          \quad\text{for $p_j\in \bZ_{\ge 0}$ depending on $J$}\\
        &=
          k_{\phi(\sigma^{-1} D\sigma)}(x)
          =k_{\phi\left(\diag(\lambda_{\sigma(1)}, \ldots ,\lambda_{\sigma(n)})\right)}(x)
          \quad\left(\text{$\Ad$-invariance of $k_{\phi(\cdot)}$}\right)\\
        &=
          x^{m-\sum p_j}\prod_{j\not\in J}(x-\lambda_{\sigma (j)})^{p_j}
          \quad\text{(because $\sigma\in S_{n,(\ca{D}_n^{= k})'_{J}}$)},
      \end{aligned} 
    \end{equation*}
    where, as before, we identified $S_n$ with  permutation matrices in $\GL(n)$. 
    This being valid for \emph{every} permutation $\sigma$ of $[n]\setminus J$ (regarded as a permutation of $[n]$ fixing $J$ pointwise) and arbitrary non-zero $\lambda_j$, $j\not\in J$, we conclude that all $p_j$, $j\not\in J$, must be equal. 
    
    
    The argument thus far ensures the existence of (perhaps $J$-dependent) non-negative integers $p=p_J$ satisfying \eqref{eq:somep} for $D\in (\ca{D}_n^{= k})'_{J}$ respectively. To check that all $p_J$ are equal, fix a diagonal matrix
    \begin{equation*}
      D:=\diag(\lambda_1, \ldots, \lambda_n)\in (\ca{D}_n^{=k })'_{J}
      ,\quad
      \lambda_j\ne 0\text{ distinct for }j\not\in J,
    \end{equation*}
    and a permutation $\sigma\in S_n$ mapping $J$ onto $J'$ for $J,J'\in [n]^{(n-k)}$. We then have
    \begin{equation*}
      \begin{aligned}
        x^{m-p_J k}\prod_{j\not\in J}(x-\lambda_j)^{p_J}
        &=
          k_{\phi(D)}(x)\\
        &=
          k_{\phi(\sigma^{-1} D\sigma)}(x)
          =k_{\phi\left(\diag(\lambda_{\sigma(1)}, \ldots ,\lambda_{\sigma(n)})\right)}(x)
          \quad\left(\text{$k_{\phi(\cdot)}$ is $\Ad$-invariant}\right)\\
        & =
          x^{m-p_{J'} k}\prod_{j\not\in J}(x-\lambda_j)^{p_{J'}},
      \end{aligned}      
    \end{equation*}
    where the last equation follows from the fact that $\sigma(J)=J'$ and
    \begin{equation*}
      \diag(\lambda_{\sigma(1)}, \ldots , \lambda_{\sigma(n)})\in (\ca{D}_n^{= k})'_{J'}.
    \end{equation*}
  Given that the $\lambda_j$, $j\not\in J$, are assumed distinct and non-zero, the equality is sufficient to conclude that $p_J=p_{J'}$. Since $J,J'\in [n]^{(n-k)}$ were arbitrary, we are done.
    
    \smallskip
    
  \item \textbf{\ref{item:th:sing:p} $\implies$ \ref{item:th:sing:p01}.} The assumed inequality in \ref{item:th:sing:p01} ensures that the only possible $p$ in \eqref{eq:somep} are $p=0,1$, corresponding to the two options the statement lists.

    \smallskip

  \item \textbf{\ref{item:th:sing:p} $\implies$ \ref{item:th:sing:inj}.} Part \ref{item:th:sing:p01} reduces the problem to arguing that in the specified regime for $m$ and $n$ our map $\phi$ cannot take nilpotent values only. This is a simple dimension count: 
    \begin{itemize}[wide]
    \item $M_n^{\le k}$ is a complex algebraic variety of (complex) dimension $k(2n-k)$ \cite[Example 14.16]{har_ag}.

    \item The space $\ca{N}(m)$ of nilpotent $m\times m$ matrices is a complex variety of dimension $m^2-m$ (e.g. \cite[Corollary 10.2.5 and Example 10.2.6]{htt_d}).

    \item Their respective \emph{covering dimensions} \cite[Definition 1.6.7]{eng_dim} are thus $2k(2n-k)$ and $2(m^2-m)$ respectively (for analytic varieties can be stratified by manifolds \cite[Theorem 6.3.3]{kp_analytic} and the covering dimension matches that of the largest manifold stratum by \cite[Corollary 1.5.4 and Theorem 1.8.2]{eng_dim}).
      
    \item And the larger-dimensional space cannot embed topologically into the smaller \cite[Theorem 1.1.2]{eng_dim}. 
    \end{itemize}    
    
    \smallskip

  \item \textbf{\ref{item:th:sing:p} $\implies$ \ref{item:th:sing:inj.bis}.} Assume that $m=n$ and that $\phi$ is injective. 
    \begin{case}
      $k > n - \sqrt{n}$.
    \end{case}
    This follows from \ref{item:th:sing:inj}: if $k > n - \sqrt{n}$ then $2k(2n-k)>2(n^2-n)$.

    \begin{case}
      $\phi(M_n^{\le k}) \subseteq M_{n}^{\le k}$.
    \end{case}

    In light of \ref{item:th:sing:p}, we have to argue that $p=1$ is the only possibility. Lemma \ref{le:open to open} applied to the map $\phi|_{M_n^{=k}}$ shows that there exist open sets $U,V \subseteq M_n^{=k}$ such that $\phi(U) = V$. By Lemma \ref{le:Rk density argument}, $V$ contains a matrix $B=\phi(A)$, $A\in U$ with $k$ distinct non-zero eigenvalues $\lambda_1,\ldots,\lambda_k \in \C^\times$. We thus have
    \begin{equation*}
      k_A(x)=x^{n-k}\prod_{j=1}^k (x-\lambda_j)=k_B(x)=k_{\phi(A)}(x),
    \end{equation*}
    so that indeed $p=1$ (by \ref{item:th:sing:p}, if this happens for \emph{one} matrix $A$, it happens globally on $M_n^{\le k}$).  \qedhere
  \end{enumerate}
\end{proof}

\begin{remark}\label{re:smlk.lgn}
  It is perhaps worth noting that the rigidity afforded by the assumed inequalities in items \ref{item:th:sing:p01}, \ref{item:th:sing:inj} and \ref{item:th:sing:inj.bis} of Theorem \ref{th:sing} is crucial (even taking the injectivity of $\phi$ for granted): as the following result shows, injections exist for any $k$ realizing any pre-selected $p$ in Theorem \ref{th:sing}\ref{item:th:sing:p} as long as $m$ is sufficiently large. In particular, setting $p=0$ in Proposition \ref{pr:any.kp} shows that for each $k \in \Z_{\ge 1}$ there exists a real analytic embedding of $M_n^{\le k}$ into the space of $n\times n$ nilpotent matrices for all sufficiently large $n$.
\end{remark}

\begin{proposition}\label{pr:any.kp}
  Let $k,p\in \bZ_{\ge 0}$ and $m,n\in \bZ_{\ge 1}$.

  If $m\ge pn$ and either $p\ge 1$ or $p=0$ and $m^2-m\ge 2k(2n-k)+1$ there are continuous embeddings
  \begin{equation*}
   \phi:  M_n^{\le k} \to  M_m
    \quad\text{with}\quad
    k_{\phi(X)}(x)
    =
    k_X^p(x)\cdot x^{m-pn}, \quad \text{ for all } X\in M_n^{\le k}.
  \end{equation*}
\end{proposition}
\begin{proof}
We can define $\phi$ block-diagonally as
  \begin{equation*}
    \phi(X):=\diag (\underbracket{X, \ldots, X}_{\text{$p$ times}}, \psi(X))
    \in
    M_m,
  \end{equation*}
  where $\psi$ is a continuous map into $\cN(m-pn)$ for all sufficiently large $m-pn$, chosen injective if $p=0$.
  
  It remains to argue that such injections do exist if $m^2-m\ge 2k(2n-k)+1$. Indeed, the proof of Theorem \ref{th:sing}\ref{item:th:sing:inj}  observes that $M_n^{\le k}$ and $\cN(m)$ are complex algebraic varieties of respective dimensions $k(2n-k)$ and $m^2-m$. They are also \emph{affine}, i.e. definable by polynomial equations as closed subsets of $\bC^N$ for appropriate $N$, hence of the form
  \begin{equation*}
    \left\{(z_1,\ldots ,z_N)\in \bC^N\ :\ f_i(z_1, \ldots, z_N)=0,\ i\in I\right\}
  \end{equation*}
  for some family $\left(f_i\right)_{i\in I}$ of polynomials in $N$ variables. By \cite[\S V.1, Theorem 1(b)]{gr_stein} this means they are \emph{Stein spaces} in the sense of \cite[\S IV.1, Definition 1]{gr_stein}, so in particular the inequality $m^2-m\ge 2k(2n-k)+1$ implies the existence of a proper holomorphic embedding  
    \begin{equation*}
      M_n^{\le k}
      \lhook\joinrel\xrightarrow{\quad}\bC^{m^2-m}
    \end{equation*}
  by \cite[p.17, item 2.]{zbMATH03169618}. In turn, $\bC^{m^2-m}$ is real-analytically isomorphic to a ball in the $(m^2-m)$-dimensional open subvariety of $\cN(m)$ consisting of non-singular points. 
\end{proof}

We now embark on the proof of Theorem \ref{th:sing2sing}, which will be presented in a sequence of carefully structured steps, following the similar approach outlined in \cite{MR4881574}. We first introduce some auxiliary notation.
\begin{itemize}[wide]
\item For $X,Y \in M_n$, by $X \leftrightarrow Y$ and $X \perp Y$ we denote that  $XY = YX$ and $XY = YX = 0$, respectively.
\item For $k \in [n]$ denote 
\begin{equation*}\
  \Lambda^k_n := \diag(1,\ldots,k,0,\ldots,0) \in 
  (\ca{D}_n^{=k})'_{\{k+1,\ldots,n\}},
\end{equation*}
where the set $(\ca{D}_n^{=k})'_{\{k+1,\ldots,n\}}$ is as specified in \eqref{eq:def.dnj}.
\item For vectors $u,v \in \C^n$ by $u \parallel v$ we denote  that the set $\{u,v\}$ is linearly dependent. The same notation is used for matrices.
\end{itemize}
\begin{proof}[Proof of Theorem \ref{th:sing2sing}] First of all, by Theorem \ref{th:sing}\ref{item:th:sing:inj.bis}, $\phi$ preservers characteristic polynomials.

\begin{claim}\label{cl:identity on diagonals}
  Without loss of generality we can assume $\phi(\Lambda^k_n) = \Lambda^k_n$ and hence that \begin{equation}\label{eq:identity on partial diagonal}
  \phi(D) = D,  \quad \text{ for all } D \in (\ca{D}_n^{\le k})_{\{k+1,\ldots,n\}}.
  \end{equation}
\end{claim}
\begin{claimproof}
  As $\phi$ preserves characteristic polynomials, the Jordan form of $\phi(\Lambda^k_n)$ looks like the block-diagonal matrix $\diag(1,\ldots,k,J)$, where $J \in M_{n-k}$ is a nilpotent Jordan matrix. Since $\phi(\Lambda^k_n) \in M_n^{\le k}$, it follows that $J= 0$. Therefore, by using a suitable conjugation matrix, without loss of generality we can assume that $\phi(\Lambda^k_n) = \Lambda^k_n$. The claim now follows by a standard argument. Indeed, fix an arbitrary $D\in (\ca{D}_n^{= k})'_{\{k+1,\ldots,n\}}$. As already observed, the set $(\ca{D}_n^{= k})'_{\{k+1,\ldots,n\}}$ is path-connected, so we can choose a continuous path 
  \begin{equation*}
  \diag(f_1, \ldots, f_k, 0, \ldots ,0) : [0,1] \to (\ca{D}_n^{= k})'_{\{k+1,\ldots,n\}}
  \end{equation*}
  such that
  \begin{equation*}
  \diag(f_1(0), \ldots, f_k(0), 0, \ldots ,0) =D \quad \text{ and } \quad \diag(f_1(1), \ldots, f_k(1), 0, \ldots ,0) =\Lambda^k_n.
  \end{equation*}
For every $t \in [0,1]$ we have
  \begin{equation*}
  \diag(f_1(t), \ldots, f_k(t), 0 , \ldots, 0)\leftrightarrow \Lambda^k_n,
  \end{equation*}
  so (by commutativity-preserving)
  \begin{equation*}
 \phi\left(\diag(f_1(t), \ldots, f_k(t), 0 , \ldots, 0)\right) \leftrightarrow \phi(\Lambda^k_n) = \Lambda^k_n.
 \end{equation*}
  Hence, 
  \begin{equation*}\phi\left(\diag(f_1(t), \ldots, f_k(t), 0 , \ldots, 0\right)) = \begin{bmatrix} \diag(*_1, \ldots, *_k) & 0 \\ 0 & *_{(n-k)\times(n-k)}\end{bmatrix},
  \end{equation*}
  where $*_1, \ldots, *_k \in \{f_1(t), \ldots, f_k(t),0\}$ (by spectrum-preserving). In view of Lemma \ref{le:topological lemma}, the map
  $$[0,1] \to M_{k}, \qquad t \mapsto \phi\left(\diag(f_1(t), \ldots, f_k(t), 0 , \ldots, 0)\right)_{\text{upper left $k\times k$ block}}$$
  is equal to one of the continuous maps
  $$t \mapsto \diag(g_1(t), \ldots, g_k(t)),$$
  where $g_1,\ldots, g_k \in \{f_1,\ldots,f_k,0\}$. The only one of these which does not contradict $\phi(\Lambda^k_n) = \Lambda^k_n$ for $t=1$ is $t \mapsto \diag(f_1(t), \ldots, f_k(t))$. Since the rank cannot increase, the lower right $(n-k)\times (n-k)$ block has to be zero. Hence, we conclude
  \begin{equation*}
  \phi\left(\diag(f_1(t), \ldots, f_k(t), 0 , \ldots, 0)\right) = \diag(f_1(t), \ldots, f_k(t), 0 , \ldots, 0), \quad \text{ for all } t \in [0,1]\end{equation*}
  and consequently, for $t=0$, $\phi(D)=D$. Now the density of $(\ca{D}_n^{= k})'_{\{k+1,\ldots,n\}}$ in $(\ca{D}_n^{\le k})_{\{k+1,\ldots,n\}}$ finishes the proof of the claim.
\end{claimproof}
\emph{In view of Claim \ref{cl:identity on diagonals}, throughout the proof of  Theorem \ref{th:sing2sing} we assume that \eqref{eq:identity on partial diagonal} holds. In particular, $\phi(E_{jj}) = E_{jj}$ for all $1 \le j \le k$.} 

\setcounter{case}{0}

\begin{claim}\label{cl:preserves-diagonalizability-pre} Without loss of generality we can further assume that 
\begin{equation*}
\phi(D) = D, \quad  \text{for all } D \in \mathcal{D}_n^{\le k}.
\end{equation*}
\end{claim}
\begin{claimproof} 
    We show that for each $k \le K \le n$, there exists an invertible matrix $T_K \in \GL(n)$ such that
    \begin{equation}\label{eq:inductive hypothesis K}
        \phi(D) = T_KDT_K^{-1}, \quad \text{ for all } D\in (\ca{D}_n^{\le k})_{\{K+1,\ldots, n\}}.
    \end{equation}
(Throughout, if $m>n$, we assume $\{m, \ldots , n\}$ to be the empty set.)  We prove the claim by induction on $K$. The base case $k = K$ is precisely \eqref{eq:identity on partial diagonal}, where $T_k=I$. Therefore, fix some $k \le K < n$ and suppose that \eqref{eq:inductive hypothesis K} holds for that $K$. Note that without loss of generality we can assume that $T_K = I$ as well. Indeed, by \eqref{eq:identity on partial diagonal} and \eqref{eq:inductive hypothesis K} we have
    \begin{equation*}
        D=\phi(D)=T_KDT_K^{-1}, \quad \text{ for all }D \in (\ca{D}_n^{\le k})_{{\{k+1,\ldots, n\}}}.
    \end{equation*}
    In particular, $T_K$ commutes with all such matrices $D$, so $T_K$ is of the block-diagonal form \begin{equation*}
        T_K = \diag(D', V), \quad \text{ for some $D' \in \ca{D}_k$ and $V \in M_{n-k}$}.
    \end{equation*}
    By replacing $\phi$ with the map $T_K^{-1}\phi(\cdot)T_K$, the equality \eqref{eq:identity on partial diagonal} remains true, while we also have
    \begin{equation}\label{eq:inductive hypothesis K2}
        \phi(D) = D, \quad \text{ for all } D \in (\ca{D}_n^{\le k})_{\{K+1,\ldots, n\}}.
    \end{equation}
    We wish to show that \eqref{eq:inductive hypothesis K} holds for $K+1$. For $\varepsilon \in \left[ 0,\frac12\right]$, denote 
    \begin{equation*}
    \Lambda(\varepsilon) := \Lambda_n^{k-1} + \varepsilon E_{K+1,K+1} \in (\ca{D}_n^{\le k})_{\{K+2,\ldots, n\}}
    \end{equation*}
    (if $k=1$, then $\Lambda(\varepsilon) = \varepsilon E_{K+1,K+1}$). For each $p \in [K]$, via \eqref{eq:inductive hypothesis K2} we have
\begin{equation*}
\Lambda(\varepsilon) \leftrightarrow E_{pp} \implies \phi(\Lambda(\varepsilon)) \leftrightarrow E_{pp},
\end{equation*}
so we conclude that 
\begin{equation}\label{eq:definition of psi}
\phi(\Lambda(\varepsilon)) = \diag(\psi(\varepsilon), E(\varepsilon))
\end{equation}
    for some matrices $\psi(\varepsilon) \in \ca{D}_{K}$ and $E(\varepsilon) \in M_{n-K}$ dependent on $\varepsilon$. Let $\ca{S}$ be the family of functions $\left(0,\frac12\right] \to \C$ consisting of the inclusion map, as well as the constant maps equal to $c \in \{0\}\cup  [k-1]$, and define the finite family
    \begin{equation*}
      \ca{G} := \left\{g : \left(0,\frac12\right] \to \mathcal{D}_{K} \,\Bigg|\,  \text{there exist } g_1, \ldots, g_K \in\ca{S} \text{ such that } g(\cdot) = \diag(g_1(\cdot), \ldots, g_K(\cdot))\right\}
    \end{equation*}
of continuous maps which attain distinct values at each $\varepsilon \in \left(0,\frac12\right]$.
   Consider the continuous map $\psi : \left[0,\frac12\right] \to \mathcal{D}_K$ which maps each $\varepsilon \in \left[0,\frac12\right]$ to the upper-left $K \times K$ block of $\phi(\Lambda(\varepsilon))$, as defined by \eqref{eq:definition of psi}. Since $\phi$ is spectrum-preserving, for each $\varepsilon \in \left(0,\frac12\right]$ we have $\psi(\varepsilon) = g(\varepsilon)$ for some $g \in \ca{G}$ (depending on $\varepsilon$). Consequently, by Lemma \ref{le:topological lemma}, we conclude that $\psi|_{\left(0,\frac12\right]} \in \ca{G}$, so there exist $g_1,\ldots,g_K\in\ca{S}$ such that $\psi|_{\left(0,\frac12\right]} = \diag(g_1,\ldots,g_K)$.  By invoking  \eqref{eq:inductive hypothesis K2} on the matrix $\Lambda(0)$ and using \eqref{eq:definition of psi}, we obtain
    \begin{align*}
        \Lambda(0) &= \phi(\Lambda(0)) = \lim_{\varepsilon \to 0^+} \phi(\Lambda(\varepsilon)) = \lim_{\varepsilon \to 0^+} \diag(\psi(\varepsilon), E(\varepsilon))\\ & =  \lim_{\varepsilon \to 0^+} \diag(g_1(\varepsilon), \ldots, g_K(\varepsilon), E(\varepsilon)).
    \end{align*}
    Hence, 
    $$\lim_{\varepsilon \to 0^+} g_p(\varepsilon) = \Lambda(0)_{pp} = p, \quad \text{ for all } p \in [k-1]$$
    and
    $$\lim_{\varepsilon \to 0^+} g_q(\varepsilon) = \Lambda(0)_{qq} = 0, \quad \text{ for all } k \le  q \le K.$$
            By the definition of $\ca{S}$, we conclude that for all $p \in [k-1]$, $g_p$ equals the constant map $p$, while for all $k \le q \le K $, $g_q$ is either the constant map $0$, or the inclusion map $\left(0,\frac12\right] \to \C$. Since the rank of the matrix $\phi(\Lambda(\varepsilon))$ cannot exceed $k$, we therefore see that there are exactly two options, the first being that there is some $k \le q \le K$ such that
    \begin{equation}\label{eq:first option}
      \phi(\Lambda(\varepsilon))= \Lambda(0) + \varepsilon E_{qq}, \quad \text{ for all } \varepsilon \in \left(0,\frac12\right],
    \end{equation}
   while the second is
    \begin{equation}\label{eq:second option}
      \phi(\Lambda(\varepsilon)) =\Lambda(0) + \diag(0_K,E(\varepsilon)), \quad \text{ for all } \varepsilon \in \left(0,\frac12\right],
    \end{equation}
   where $E(\varepsilon)\in M_{n-K}$ is a rank-one matrix with $\varepsilon$ being its only non-zero eigenvalue (hence it is diagonalizable). In view of  \eqref{eq:inductive hypothesis K2}, the first option \eqref{eq:first option} clearly contradicts the injectivity of $\phi$. Therefore, the second option \eqref{eq:second option} must be true. In particular, for $\varepsilon = \frac12$, let $S \in \GL(n-K)$ be an invertible matrix such that
   \begin{equation*}
       SE\left(\frac12\right)S^{-1} = \diag\left(\frac12, 0, \ldots, 0\right) \in \ca{D}_{n-K}.
   \end{equation*}
   By conjugating $\phi$ with $\diag(I_K,S)$ (so that \eqref{eq:inductive hypothesis K2}, and in particular \eqref{eq:identity on partial diagonal}, still holds), we can further assume that
    \begin{equation}\label{eq:lambda1/2}
\phi\left(\Lambda\left(\frac12\right)\right)=\Lambda\left(\frac12\right).
    \end{equation}
    Let $D \in \ca{D}_n^{\le k}$ be an arbitrary matrix supported in $[K+1]$. In view of \eqref{eq:inductive hypothesis K2}, for each $p \in [K]$ we have
    $$D \leftrightarrow E_{pp} \implies \phi(D) \leftrightarrow E_{pp}$$
    and
    \begin{equation*}
    D \leftrightarrow \Lambda\left(\frac12\right) \implies \phi(D) \leftrightarrow \Lambda\left(\frac12\right)
    \end{equation*}
    so that $\phi(D) \leftrightarrow E_{K+1,K+1}$. We conclude that there exists a diagonal matrix $\widetilde{\psi}(D) \in \ca{D}_{K+1}$, and a matrix $\widetilde{E}(D) \in M_{n-(K+1)}$ such that
    \begin{equation}\label{eq:defintion of psi tilda}
        \phi(D) = \diag(\widetilde{\psi}(D), \widetilde{E}(D))
    \end{equation}
    (if $K+1 = n$, then simply $\phi(D) = \widetilde{\psi}(D)$). This equality clearly defines a continuous map $\widetilde{\psi} : (\ca{D}_n^{\le k})_{\{K+2,\ldots, n\}} \to \ca{D}_{K+1}$. Denote \begin{equation*}
        j := K+1-k \ge 1
    \end{equation*} and let $R \in [K]^{(j)}$ be an arbitrary choice of $j$ distinct indices. Let $\widetilde{\ca{G}}_R$ denote the finite family consisting of all functions $f_\pi: (\ca{D}_n^{=k})'_{R \cup \{K+2,\ldots, n\}} \to \ca{D}_{K+1}$, parameterized by maps $\pi : [K+1] \to ([K+1]\setminus R) \cup \{0\}$, defined by
    \begin{equation*}
        f_{\pi}(D) = \diag(d_{\pi(1)}, \ldots, d_{\pi(K+1)}), \quad \text{ for all }D = \diag(d_1,\ldots,d_n) \in  (\ca{D}_n^{=k})'_{R \cup \{K+2,\ldots, n\}},
    \end{equation*}
    where $d_0$ is defined to be $0$. Note that $\widetilde{\ca{G}}_R$ satisfies the assumptions of Lemma \ref{le:topological lemma}. Indeed, clearly every function in $\widetilde{\ca{G}}_R$ is continuous. Moreover, fix some $D = \diag(d_1,\ldots,d_n) \in  (\ca{D}_n^{=k})'_{R \cup  \{K+2,\ldots, n\}}$ and let $\pi,\pi' : [K+1] \to ([K+1]\setminus R) \cup \{0\}$ be maps such that there exists some $i \in [K+1]$ with $\pi(i) \ne \pi'(i)$. But then
    $$f_{\pi}(D)_{ii} = d_{\pi(i)} \ne d_{\pi'(i)} = f_{\pi'}(D)_{ii} \implies f_{\pi}(D) \ne f_{\pi'}(D),$$
    as desired. Furthermore, as $\phi$ is spectrum-preserving, note that for any $D = \diag(d_1,\ldots,d_n) \in  (\ca{D}_n^{=k})'_{R \cup \{K+2,\ldots, n\}}$ there clearly exists a map $\pi : [K+1] \to ([K+1]\setminus R) \cup \{0\}$ such that $\widetilde{\psi}(D) = f_{\pi(D)}$. As the space $(\ca{D}_n^{=k})'_{R \cup \{K+2,\ldots, n\}}$ is (path-)connected, we can therefore apply Lemma \ref{le:topological lemma} to the continuous restriction $$\widetilde{\psi}_R := \widetilde{\psi}|_{(\ca{D}_n^{=k})'_{R \cup \{K+2,\ldots, n\}}} : (\ca{D}_n^{=k})'_{R \cup\{K+2,\ldots, n\}} \to \ca{D}_{K+1}$$ to conclude that $\widetilde{\psi}_R = f_{\pi_R} \in \widetilde{\ca{G}}_R$ for some map $\pi_R  : [K+1] \to ([K+1]\setminus R) \cup \{0\}$. Fix some $p \in [K] \setminus R$. By the density of $(\ca{D}_n^{=k})'_{R \cup\{K+2,\ldots, n\}}$ in $(\ca{D}^{\le k}_n)_{R \cup \{K+2,\ldots, n\}}$, there exists a sequence $(D_m)_{m\in \N}$ in $(\ca{D}_n^{=k})'_{R \cup\{K+2,\ldots, n\}}$ such that $\lim_{m\to\infty} D_m = E_{pp}$. Then, by \eqref{eq:inductive hypothesis K2}, we have
    \begin{align}\label{eq:limits Epp}
        1 &= (E_{pp})_{pp}= \widetilde{\psi}(E_{pp})_{pp} = \lim_{m\to\infty} \widetilde{\psi}_R(D_m)_{pp} = \lim_{m\to\infty} (D_m)_{\pi_R(p),\pi_R(p)} \\
        &= (E_{pp})_{\pi_R(p),\pi_R(p)}. \nonumber
    \end{align}
    Therefore, $\pi_R(p) = p$ for each $p \in [K]\setminus R$.
    
    
For a moment consider the particular choice $R_0:= \{k,\ldots, K\}$ (which is in fact the only possibility when $k=1$). For the  matrix $\Lambda\left(\frac12\right) \in (\ca{D}_n^{=k})'_{R_0 \cup \{K+2,\ldots, n\}}$, by \eqref{eq:lambda1/2} we know that
    \begin{equation*}
        \widetilde{\psi}_{R_0}\left(\Lambda\left(\frac12\right)\right) = \diag\left(1,\ldots,k-1,0,\ldots,0,\frac12\right),
    \end{equation*} so 
    \begin{equation*}
       \frac12 = \widetilde{\psi}_{R_0}\left(\Lambda\left(\frac12\right)\right)_{K+1,K+1} = f_{\pi_{R_0}}\left(\Lambda\left(\frac12\right)\right)_{K+1,K+1} = \Lambda\left(\frac12\right)_{\pi_{R_0}(K+1),\pi_{R_0}(K+1)}.
    \end{equation*}
    It follows that $\pi_{R_0}(K+1) = K+1$, i.e.\ the restriction $\pi_{R_0}|_{[K+1]\setminus R_0}$ acts as the identity map on $[K+1]\setminus R_0$. Therefore, 
    \begin{equation*}
        \widetilde{\psi}_{R_0}(D)_{pp} = D_{pp}, \quad \text{ for all $D \in  (\ca{D}_n^{=k})'_{R_0 \cup \{K+2,\ldots, n\}}$ and $p \in [K+1]\setminus R_0$}.
    \end{equation*}
    Since for the arbitrary matrix $D$ as above, the rank of $\phi(D) = \diag(\widetilde{\psi}(D), \widetilde{E}(D))$ cannot exceed $k$, we conclude that $\widetilde{\psi}_{R_0}(D)$ is precisely the upper-left $(K+1)\times(K+1)$ block of $D$, while $\widetilde{E}(D) = 0$. By passing to the closure, it follows that  $\phi(D) = D$ for all $D \in \ca{D}_n^{\le k}$ supported in $[k-1] \cup \{K+1\}$ (which consequently finishes the proof when $k=1$). In particular, we have
    \begin{equation}\label{eq:EK+1}
        \phi(E_{K+1, K+1}) = E_{K+1,K+1}.
    \end{equation}

     We now return to the arbitrary choice of $R$. Let $(D_m)_{m\in \N}$ be a sequence of matrices in $(\ca{D}_n^{=k})'_{R\cup \{K+2,\ldots, n\}}$ such that $\lim_{m\to\infty} D_m = E_{K+1,K+1}$. By the same reasoning as in \eqref{eq:limits Epp}, using \eqref{eq:defintion of psi tilda} and \eqref{eq:EK+1} we have
    \begin{align*}
        1 &= (E_{K+1,K+1})_{K+1,K+1}= \widetilde{\psi}(E_{K+1,K+1})_{K+1,K+1} = \lim_{m\to\infty} \widetilde{\psi}_R(D_m)_{K+1,K+1} \\
        &= \lim_{m\to\infty} (D_m)_{\pi_R(K+1),\pi_R(K+1)} = (E_{K+1,K+1})_{\pi_R(K+1),\pi_R(K+1)},
    \end{align*}
    which directly implies $\pi_R(K+1) = K+1$. It follows that $\pi_R|_{[K+1]\setminus R}$ acts as the identity map on $[K+1]\setminus R$. As in the case $R = R_0$, using the rank argument we conclude that
    $$\phi(D) = D, \quad \text{ for all }D \in (\ca{D}_n^{=k})'_{R\cup \{K+2,\ldots, n\}} \text{ and any choice of $R \in [K]^{(j)}$}.$$
    By passing to the closure, it follows
    $$\phi(D) = D, \quad \text{ for all $D \in \ca{D}_n^{\le k}$ supported in }[K+1].$$
    This finally completes the inductive step.
\end{claimproof}

The rest of the argument follows using the simplified arguments from \cite{MR4881574}. For completeness we include some details. 

\begin{claim}\label{cl:preserves-diagonalizability}
  For each $S \in \GL(n)$ there exists $T \in \GL(n)$ such that \begin{equation*}\phi(SDS^{-1}) = TDT^{-1}, \quad \text{ for all } D \in \mathcal{D}_n^{\le k}.\end{equation*}
\end{claim}
\begin{claimproof}
  For a fixed $S \in \GL(n)$ there exists $T \in \GL(n)$ such that $\phi(S\Lambda_n^k S^{-1}) = T \Lambda_n^k T^{-1}$. Now we can apply Claim \ref{cl:identity on diagonals} to the map $T^{-1}\phi(S(\cdot)S^{-1})T$ which satisfies the same properties as $\phi$, as well as $\Lambda_n^k \mapsto \Lambda_n^k$.
\end{claimproof}

\begin{claim}\label{cl:preserves zero-product}
  Let $A,B \in M_n^{\le k}$ be two diagonalizable matrices such that $A \perp B$. Then $\phi(A) \perp \phi(B)$.
\end{claim}
\begin{claimproof}
  Follows directly from Claim \ref{cl:preserves-diagonalizability}.
\end{claimproof}

\begin{claim}\label{cl:phi is homogeneous}\phantom{x}
  \begin{enumerate}[(a)]
  \item Let $A,B \in M_n^{\le k}$ be diagonalizable matrices such that $A \leftrightarrow B$. Then $\phi(\alpha A + \beta B) = \alpha \phi(A) + \beta\phi(B)$ for all $\alpha,\beta \in \C$.
  \item $\phi$ is a homogeneous map.
  \end{enumerate}
\end{claim}
\begin{claimproof}
  (a) follows directly from  and Claim \ref{cl:preserves-diagonalizability}, while (b) follows from (a), Lemma \ref{le:Rk density argument} and the continuity of $\phi$.
\end{claimproof}

\begin{claim}\label{cl:preserves rank-one}
  $\phi$ is a rank-decreasing map. In particular, $\phi(M_n^{=1}) \subseteq M_n^{=1}$.
\end{claim}
\begin{claimproof}
  In view of Claim \ref{cl:preserves-diagonalizability}, $\phi$ in fact preserves the rank of diagonalizable matrices. On the other hand, every matrix in $M_n^{\le k}$ is, by Lemma \ref{le:Rk density argument}, a limit of a sequence of diagonalizable matrices with the same rank. Now the claim follows from the continuity of $\phi$ and the lower semicontinuity of the rank. 

\end{claimproof}

\begin{claim}\label{cl:n=3}
  The theorem is true when $n = 3$.
\end{claim}
\begin{proof}
  This follows from the same arguments in \cite[Section~2]{PetekSemrl}, which are conducted entirely within $M_{3}^{\leq 1}$. Indeed, the rank-one preservation assumption in \cite{PetekSemrl} serves to preclude the case of $\phi(E_{ij}) = 0$ for some matrix unit $E_{ij}$, while here this is a consequence of injectivity.
\end{proof}

\begin{claim}\label{cl:preserves orthogonality}
  Suppose that matrices $A_1,A_2 \in M_n^{\le 1}$ satisfy $A_1 \perp A_2$. Then $\phi(A_1) \perp \phi(A_2)$.
\end{claim}
\begin{claimproof}
  This follows exactly as in \cite[Claim 7]{MR4881574}.

\end{claimproof}

\begin{claim}\label{cl:parallel}
  We have \begin{equation*}\phi(E_{ij}) \parallel E_{ij}, \quad \text{ for all } (i,j) \in [n]^2\end{equation*} or \begin{equation*}\phi(E_{ij}) \parallel E_{ji},\quad  \text{ for all } (i,j) \in [n]^2.\end{equation*}
\end{claim}
\begin{claimproof}
  For each $(i,j) \in [n]^2$ we have $E_{ij}\perp E_{kk}$ for all $k \in [n]\setminus \{i,j\}$ so by Claim \ref{cl:preserves orthogonality} we obtain \begin{equation*}\supp \phi(E_{ij}) \subseteq \{i,j\} \times \{i,j\}.\end{equation*}
  Suppose that $(i,j),(i,k) \in [n]^2\setminus \{(1,1),\ldots,(n,n)\}$ for $j \ne k$. By arguments similar to \cite[Claim 8]{MR4881574}, it is not difficult to show that we have either
  \begin{equation*}
    \phi(E_{ij}) \parallel E_{ij} \quad \text{ and }  \quad \phi(E_{ik}) \parallel E_{ik}
  \end{equation*}
  or
  \begin{equation*}
    \phi(E_{ij}) \parallel E_{ji}  \quad \text{ and }  \quad \phi(E_{ik}) \parallel E_{ki}.
  \end{equation*}

  \smallskip

  \noindent Now suppose $(i,j),(k,j) \in [n]^2\setminus \{(1,1), \ldots , (n,n)\}$ for $i \ne k$. By Claim \ref{cl:preserves orthogonality} we have
  \begin{equation*}
    E_{ij} \perp E_{kj} \implies \phi(E_{ij}) \perp \phi(E_{kj})
  \end{equation*}
  so we conclude that either
  \begin{equation*}
    \phi(E_{ij}) \parallel E_{ij} \quad \text{ and } \quad \phi(E_{kj}) \parallel E_{kj}
  \end{equation*}
  or
  \begin{equation*}
    \phi(E_{ij}) \parallel E_{ji} \quad \text{ and } \quad \phi(E_{kj}) \parallel E_{jk}.
  \end{equation*}
  Therefore, $\phi$ behaves the same way on all matrix units which are contained in the same row or in the same column. Now it is easy to conclude that $\phi$ behaves the same way globally on all matrix units, which is what we wanted to prove.
\end{claimproof}

By passing to the map $\phi(\cdot)^t$ if necessary, we can assume that
\begin{equation*}\phi(E_{ij}) \parallel E_{ij}, \quad \text{ for all } (i,j) \in [n]^2.\end{equation*}
In view of Claim \ref{cl:parallel}, for each $(i,j) \in [n]^2$, denote by $g(i,j) \in \C^\times$ the unique scalar such that \begin{equation*}\phi(E_{ij}) = g(i,j)E_{ij} \quad\text{ or }\quad\phi(E_{ij}) = g(i,j)E_{ji}.\end{equation*}
As, by assumption, $\phi|_{\ca{D}_n^{\le k}}$ is the identity map, it is immediate that $g|_{\{(1,1), \ldots , (n,n)\}} \equiv 1$.

\smallskip

\noindent Following the notation established in \cite{MR4881574}, for  $S \subseteq [n]$ and $A \in M_n$, we define the following auxiliary notation:
\begin{itemize}[wide]
\item $M_n^{\subseteq S}:= \{X \in M_n \ : \ \supp X \subseteq S\times S\}$.
\item When $S \ne [n]$, denote by $A^{\flat S} \in M_{n - \abs{S}}$ the matrix obtained from $A$ by deleting all rows $i$ and columns $j$ where $i,j \in S$. We also formally allow $A^{\flat \emptyset} = A$.
\item Denote by $A^{\sharp S} \in M_{n + \abs{S}}$ the matrix obtained from $A$ by adding zero rows and columns so that $(A^{\sharp S})^{\flat S} = A$.
\end{itemize}
By using block-matrix multiplication, it is not difficult to verify that $M_n^{\subseteq S}$ is a subalgebra of $M_n$, and that $(\cdot)^{\flat ([n] \setminus S)} : M_n^{\subseteq S} \to M_{\abs{S}}$ and $(\cdot)^{\sharp S} : M_n \to M_{n + \abs{S}}$ are algebra monomorphisms (see \cite[Lemmas~2.4 and 2.5]{MR4881574}).  We also extend the notation $(\cdot)^{\flat S}$ and $(\cdot)^{\sharp S}$ to sets of matrices by applying the respective operation element-wise.

\begin{claim}\label{cl:phi preserves support} 
  Suppose that $X \in M_n^{\le k}$ satisfies $X \in M_n^{\subseteq S}$ for some $S \subseteq [n]$. Then $\phi(X) \in M_n^{\subseteq S}$.
\end{claim}
\begin{claimproof}
  By applying Lemma \ref{le:Rk density argument} to the matrix $X^{\flat ([n] \setminus S)} \in M_{\abs{S}}^{\le k}$, we can approximate it by a diagonalizable matrix $Y\in M_{\abs{S}}^{\le k}$ of the same rank. Then $Y^{\sharp([n] \setminus S)} \in M_n^{\le k} \cap M_n^{\subseteq S}$ approximates $X$ and is diagonalizable. In view of the continuity of $\phi$ it therefore suffices to assume that $X$ itself is already diagonalizable. Then the assertion follows from $X \perp E_{kk}$ for all $k \in [n]\setminus S$,  Claim \ref{cl:preserves zero-product} and the fact that $\phi(E_{kk}) = E_{kk}$.
\end{claimproof}

\begin{claim}\label{cl:restriction of phi}
  Let $S \subseteq [n]$ be a nonempty set. The map
  \begin{equation*}\psi : (M_n^{\le k})^{\flat ([n] \setminus S)} = M_{\abs{S}}^{\le k} \to M_{\abs{S}}^{\le k}, \qquad  X \mapsto \phi(X^{\sharp ([n] \setminus S)})^{\flat ([n] \setminus S)}\end{equation*}
  is an injective continuous commutativity and spectrum preserver.
\end{claim}
\begin{claimproof}
  This is easy to verify directly (see the proof of \cite[Claim~11]{MR4881574}).
\end{claimproof}

Arguments similar to ones immediately after \cite[Claim 11]{MR4881574} show that
\begin{equation*}
g(i,j)g(j,k) = g(i,k), \quad \text{ for all } (i,j),(j,k) \in [n]^2.
\end{equation*}
Now it is easy to see that for \begin{equation*}
D := \diag(g(1,1), \ldots, g(1,n)) \in \mathcal{D}_n \cap \GL(n) 
\end{equation*} 
we have
\begin{equation*}\phi(E_{ij}) = D^{-1}E_{ij}D, \quad \text{ for all } (i,j) \in [n]^2\end{equation*}
so by passing to the map $D\phi(\cdot)D^{-1}$ without loss of generality we will assume that
\begin{equation*}\phi(E_{ij}) = E_{ij}, \quad \text{ for all } (i,j) \in [n]^2.\end{equation*}

\begin{claim}\label{cl:identity on R}
  $\phi$ acts as the identity on all rank-one non-nilpotents, i.e.\ $(M_{n}^{= 1})'_{s}$.
\end{claim}
\begin{claimproof}
  Fix some $ab^* \in (M_{n}^{=1})'_{s}$ for some non-zero vectors $a,b \in \C^n$. In view of Claim \ref{cl:preserves rank-one}, denote $\phi(ab^*) = xy^*$ for some non-zero $x,y \in \C^n$. Since $b^*a = \Tr (ab^*) \ne 0$, we can choose some $j \in (\supp a) \cap (\supp b)$.   
  Fix some distinct $i \in [n] \setminus \{j\}$ and consider
  \begin{equation*}
    A := (\overline{b_j}e_i - \overline{b_i}e_j)(\overline{a_i}e_j - \overline{a_j}e_i)^* \in M_n^{\le 1}.
  \end{equation*}
  Clearly, $\supp A \subseteq \{i,j\}\times\{i,j\}.$ Since $n \ge 3$, choose some $k \in [n]\setminus \{i,j\}$. We can now invoke Claim \ref{cl:restriction of phi} and the $n=3$ case (Claim \ref{cl:n=3}) to conclude that the map
  \begin{equation*}\psi : (M_n^{\le k})^{\flat([n] \setminus \{i,j,k\})} = M_3^{\le k} \to M_3^{\le k}, \qquad X \mapsto \phi(X^{\sharp ([n] \setminus \{i,j,k\})})^{\flat([n] \setminus \{i,j,k\})}\end{equation*}
  is of the form \eqref{eq:inner}. Since $\psi$ acts as the identity on all matrix units of $M_3^{\le k}$, we conclude that $\psi$ is the identity map. In particular, Claim \ref{cl:phi preserves support} implies that $\phi(A) = A$.

  \smallskip
  
  Now notice that $ab^* \perp A$, so by Claim \ref{cl:preserves orthogonality} we obtain
  \begin{equation*}xy^* = \phi(ab^*) \perp \phi(A) = A \implies (\overline{a_i}e_j - \overline{a_j}e_i)^*x = y^*(\overline{b_j}e_i - \overline{b_i}e_j) = 0.\end{equation*}
  Overall, it follows
  \begin{equation*}x \perp \ca{B}_1:=\{\overline{a_i}e_j - \overline{a_j}e_i \ : \ i \in [n]\setminus \{j\}\} \quad \text{ and } \quad y \perp \ca{B}_2:=\{\overline{b_j}e_i - \overline{b_i}e_j \ : \ i \in [n]\setminus \{j\}\}.\end{equation*}
In fact, the sets $\ca{B}_1$ and $\ca{B}_2$ are bases for $\{a\}^\perp$ and $\{b\}^\perp$, respectively. We conclude $x\parallel a$ and $y \parallel b$, which implies $\phi(ab^*) \parallel ab^*$. Equating the traces yields $\phi(ab^*) = ab^*$.
\end{claimproof}

\begin{claim}\label{cl:penultimate step}
  $\phi$ is the identity map.
\end{claim}
\begin{claimproof}
  The proof is similar to that of \cite[Claim 15]{MR4881574}. More specifically, by Claim \ref{cl:preserves-diagonalizability}, for each $S \in \GL(n)$ there exists $T \in \GL(n)$ such that
  \begin{equation*}\phi(SDS^{-1}) = TDT^{-1}, \quad \text{ for all } D \in \ca{D}_n^{\le k}.\end{equation*}
  In particular, for all $j \in [n]$ we have
  \begin{equation*}
    TE_{jj}T^{-1} = \phi(\underbracket{SE_{jj}S^{-1}}_{\in M_n^{\le 1}}) \stackrel{\text{Claim } \ref{cl:identity on R}}= SE_{jj}S^{-1}.
  \end{equation*} Hence, by the linearity of the maps $T(\cdot)T^{-1}$ and $S(\cdot)S^{-1}$, for all $D \in \ca{D}_n^{\le k}$ we have
  \begin{equation*}
    \phi(SDS^{-1}) = TDT^{-1} = SDS^{-1}.
  \end{equation*}
  The Claim now follows from Lemma \ref{le:Rk density argument} and the continuity of $\phi$.
\end{claimproof}
\end{proof}
We conclude the paper with a brief discussion on the necessity of assumptions in Theorem \ref{th:sing2sing}.
\begin{remark}\label{re:indispensable}
  \begin{enumerate}[wide]
  \item The assertion of Theorem \ref{th:sing2sing} does not hold for $n=2$. This follows from \cite[Proof of Case 1 in (ii) $\implies$ (i) of Theorem~3.7]{MR4881574}, which is a slight modification of \cite[Example 7]{PetekSemrl} (in order to ensure the injectivity). More specifically, let $f : [0,+\infty) \to \Sph^1$  be a  function defined by $f(t):= e^{\frac{i\pi}{t+1}}$. Define $\phi : M_2^{\le 1} \to M_{2}^{\le 1}$  by
    \begin{equation*}
      \phi\left(\begin{bmatrix}
        a & b \\ c & d
      \end{bmatrix}\right)
    :=
    \left\{\begin{array}{cc}
      \begin{bmatrix}
        a & 0 \\ c & d
      \end{bmatrix}, & \text{ if } b= 0, \\
      \begin{bmatrix}
        a & b\,f\left(\abs{\frac{c}b}\right) \\ c\,\overline{f\left(\abs{\frac{c}b}\right)} & d
      \end{bmatrix}, & \text{ otherwise.}
    \end{array}\right.
\end{equation*}
Then $\phi$ is an injective continuous spectrum and commutativity preserver. On the other hand, 
\begin{equation*}
  \phi\left(\begin{bmatrix}
    1 & 1 \\ 1 & 1
  \end{bmatrix}\right)=
\begin{bmatrix}
  1 & i \\ -i & 1
\end{bmatrix}
\quad\text{ and } \quad
\phi\left(\begin{bmatrix}
  1 & 1 \\ 0 & 0
\end{bmatrix}\right)
+ \phi\left(\begin{bmatrix}
  0 & 0 \\ 1 & 1
\end{bmatrix}\right)
=\begin{bmatrix}
  1 & -1 \\ 1 & 1
\end{bmatrix},
\end{equation*}
which clearly shows that $\phi$ is not of the form \eqref{eq:inner}.
\end{enumerate}
\smallskip
Now let $n \ge 3$ and $1 \le k < n$.
\begin{enumerate}[wide] \setcounter{enumi}{1}
\item The map $\phi : M_{n}^{\le k} \to M_{n}^{\le k}$ given by $\phi(X) = 2X$ is commutativity-preserving, injective and continuous, but not spectrum-shrinking.
\item Following \cite[Section~4]{Semrl} (see also \cite[Remark~4.3]{MR4881574}), the map $\phi : M_{n}^{\le k} \to M_{n}^{\le k}$ given by $\phi(X) = f(X)Xf(X)^{-1}$, for some continuous map $f : M_n^{\le k} \to \GL(n)$ is injective, continuous and spectrum-preserving, but not commutativity-preserving in general.
\item The map $\phi : M_{n}^{\le k} \to M_{n}^{\le k}$ given by 
  \begin{equation*}
    \phi(X) = \begin{cases}
      -X, \qquad &\text{ if }X \text{ is nilpotent} ,\\
      X, &\text{ otherwise}
    \end{cases}
  \end{equation*}
  is injective and preserves spectra and commutativity, but is not continuous.
\item The map $\phi: M_n^{\leq k} \to M_n^{\le k}$ that sends every matrix to a fixed rank-one nilpotent matrix is clearly continuous, commutativity-preserving and spectrum-shrinking, but not injective.
\end{enumerate}
\end{remark}



\def\polhk#1{\setbox0=\hbox{#1}{\ooalign{\hidewidth
  \lower1.5ex\hbox{`}\hidewidth\crcr\unhbox0}}}

\end{document}